\documentclass{aptpub}

\usepackage{bbm}
\usepackage{datetime}
\usepackage{enumerate}
\usepackage{xcolor}
\usepackage{graphicx}
\usepackage{epstopdf}
\usepackage{url}

\usepackage{setspace}
\numberwithin{equation}{section}
\theoremstyle{plain}

\newcommand{\cid}{\stackrel{d}{\longrightarrow}}

\newcommand{\civ}{\stackrel{v}{\longrightarrow}}
\newcommand{\toi}{\to\infty}
\newcommand{\eind}{\stackrel{d}{=}}

\newcommand{\dint}{\displaystyle\int}

\authornames{BASRAK AND \v SPOLJARI\'C} 
\shorttitle{On randomly spaced observations} 

\begin{document}

\title{On randomly spaced observations and continuous time random walks} 

\authorone[University of Zagreb]{Bojan Basrak} 
\addressone{Department of Mathematics, University of Zagreb, Bijeni\v cka 30, Zagreb, Croatia} 

\authortwo[University of Zagreb]{Drago \v Spoljari\'c}
\addresstwo{Faculty of Mining, Geology and Petroleum Engineering, University of Zagreb, Pierottijeva 6, Zagreb, Croatia}

\begin{abstract}


We consider random variables observed at arrival times of a renewal process, which possibly depends on those observations and has regularly varying steps with infinite mean. 
Due to the dependence and heavy tailed steps,
the limiting behavior of extreme observations until a given time $t$
 tends to be rather involved. 
We describe this asymptotics 
and  generalize several partial results which appeared in this
setting. In contrast to the earlier studies, our analysis is based in the point processes theory. The theory is applied to determine the asymptotic distribution of maximal excursions and sojourn times for continuous time random walks.

\end{abstract}



\section{Introduction} 

A sequence of 
 observations collected  at arrival times of a renewal process
represents a popular modeling framework in applied probability.  
  It underlies many standard stochastic models, ranging from
risk theory and engineering to theoretical physics.
The asymptotic distribution of the maximum of such observations until a given time $t$
is well understood in the case of the renewal process with finite mean interarrival times, see  Basrak and \v Spoljari\'c \cite{BSSPL} and references therein. 
Shanthikumar and Sumita in \cite{SS83} and Anderson in \cite{Anderson87}
considered the problem in the infinite mean case, allowing
certain degree of dependence between the observations and the interarrival  times.
In such a setting,
the asymptotic theory for partial maxima tends to be
more complicated.

It is often essential to understand the distribution of all 
extreme observations and not merely the partial maximum until a given time.
Therefore, we aim to characterize the limiting behavior of all 
upper order statistics in a sequence $(X_n)$ up to a time $\tau(t)$, where $(\tau(t))$ represents the renewal process generated by a sequence of nonnegative (and nontrivial) random variables $(Y_n)$, i.e. 
\begin{equation}\label{def_renewal_process}
\tau(t) = \inf \{ k: Y_1+\cdots + Y_k >t \}\,,
\quad \mbox{ for } t\geq 0\,.
\end{equation}
Throughout we shall assume that $(X_n,Y_n),\, n\geq 1,$ form iid pairs and that the distribution of $X_1$ belongs to 
the maximum domain of attraction (MDA for short) of one of the three extreme value distributions, denoted by $G$. Moreover, the interarrival times
are assumed to be regularly varying with index $\alpha \in (0,1)$
(as in ~\cite{Anderson87,MS2004,PMM2009}).
In particular, $Y_i$'s have infinite mean and belong to the MDA of 
a Fr\'echet distribution themselves.

To explain the  limiting behavior of all large values in the sequence $(X_n)$, which
arrive before time  $t$, we rely on the theory of point processes.
Such an approach seems to be new in this context. 
It does not only yield more general results, but
we believe, it provides a very good insight into why all the previously established results actually hold.
 Moreover, we  relax restrictions
 used in the literature concerning dependence
between the observations and interarrival times.

We apply our results to study
the continuous time random walk (CTRW),  introduced by Montroll and Weiss~\cite{Montroll65}. It is essentially  a random walk subordinated to a renewal process.
It has numerous applications in physics  and has been used to model various phenomena in finance, hydrology, quantum transport theory and seismology. For an overview of the literature on the theory and various applications of CTRW we refer to ~\cite{M3ExtremalCCTRW2011}.
Under mild  assumptions, excursions of such a walk are regularly
varying with index $1/2$. Hence, our theory applies, and one can 
determine  the limiting distribution of extremely long excursions and 
sojourn times at level zero of a CTRW.

The paper is organized as follows: notation and auxiliary results are introduced in section~\ref{section_prelims}. In section~\ref{Sec:Ext}, we present the limiting distribution
of all extreme observations  until a given time, discussing
 asymptotic tail independence
and asymptotic full tail dependence between the observations and interarrival times in detail.
Our main result extends previously mentioned results in this context,  as well as  some more recent results in  Meerschaert and Stoev~\cite{MStoev2009} and Pancheva et al.~\cite{PMM2009} for instance. 
 In the asymptotic full tail dependence case. our result could  be  applied to study the longest intervals of
the renewal process $\tau$ itself,  which is the subject of  recent
papers by Goldr\`eche et al.~\cite{GMS} and Basrak~\cite{BB15}.
In section~\ref{Sec:CTRW} we apply the main result to study the longest sojourn times and the longest excursions of a continuous time random walk. In particular, this section extends the analysis of asymptotic distribution for the  ranked lengths of excursions in the simple symmetric random walk given by
   Cs\'aki and Hu~\cite{Csaki2003}.

%

\section{Preliminaries}\label{section_prelims}

As already mentioned in the introduction, we assume that the
observations $X_i$ belong to MDA of some extreme value distributions, denoted by $G$. Because of the correspondence between MDA's of Fr\'echet and Weibull distributions, we discuss only observations in the Gumbel (which is denoted by $G=\Lambda$) and Fr\'echet ($G=\Phi_\beta$, for $\beta>0$) MDA's in detail (see subsection 3.3.2 in Embrechts et al.~\cite{EKMikosch}).
 In particular, there exist functions $a(t)$ and $b(t)$ such that 
\begin{equation}\label{MDA_condition}
 t P( X_1 > a(t) x + b(t)) \to -\log G(x)\,, 
\end{equation}
as $t\toi$ (cf. Resnick~\cite{ResEVRVPP}).

Recall next that we assumed that the renewal steps $Y$ have regularly varying distribution of  infinite mean with index $\alpha\in(0,1)$ (we denote this by $Y\sim\mathit{RegVar}(\alpha)$). In such a case, it is well known (see Feller~\cite{FellerVol2}) that there exists a strictly positive sequence $(d_n)$ such that
\[
d_n^{-1}(Y_1+\cdots +Y_n)\cid S_\alpha\,,
\]
where random variable $S_\alpha$ has the stable law with the index $\alpha$, scale parameter $\sigma=1$, skewness parameter $\beta=1$ and  shift parameter $\mu=0$. In particular, $S_\alpha$ is strictly positive a.s. The sequence $(d_n)$ can be chosen such that 
\begin{equation}\label{def_d(t)}
n(1-F_Y(d_n))\to 1\,,
\end{equation}
as $n\toi$, where $F_Y$ denotes cdf of $Y_1$.  
Denote $d(t)=d_{\lfloor t\rfloor}$, for $t \geq 0$, with $d_0=0$ and observe that the function $d$ is regularly varying with index $1/\alpha$. Consider the partial sum process
\begin{equation}\label{def_T(t)}
 T(t)=\sum_{i=1}^{\lfloor t\rfloor}Y_i \,,
\end{equation}
with $T(t)=0$ for $0\leq t<1$. It is well known that 
\begin{equation}\label{def_S_alpha}
 \left(\frac{T(tc)}{d(t)}\right)_{c\geq 0} \cid (S_\alpha(c))_{c\geq 0}\,,
\end{equation}
as $t\toi$, in a space of c\`adl\`ag functions $D[0,\infty)$ endowed with Skorohod $J_1$ topology (see Skorohod~\cite{Skorohod57} or Resnick~\cite{ResHTP}). The limiting process $(S_\alpha(c))_{c \geq 0}$ is an  $\alpha$--stable process with  
strictly increasing sample paths. 

Recall that for a function $z\in D([0,\infty),[0,\infty))$, the right--continuous generalized inverse is defined by the relation
\begin{equation*}
z^{\leftarrow}(u)=\inf\{s\in[0,\infty):z(s)>u\} \,, \quad u \geq 0\,.
\end{equation*}

The following lemma  shows that under certain conditions, 
the convergence of  functions $z_t $ to $ z $ in $J_1$ topology
implies the convergence of corresponding
generalized inverses in the same topology.
The content of the lemma seems to be known (cf. Resnick~\cite[p. 266]{ResHTP} and Theorem 7.2 in Whitt~\cite{Whitt80}), 
for convenience we add a short proof.

\begin{lemma}\label{lemma_gen_inverse_cvg}
Suppose that $z_t\in D([0,\infty),[0,\infty))$ are non-decreasing functions for all $t \geq 0$. Further, let $z\in D([0,\infty),[0,\infty))$ be strictly increasing to infinity.  If
$
z_t\stackrel{J_1}{\longrightarrow} z,
$
as $t\toi$, then 
$
z_t^{\leftarrow}\stackrel{J_1}{\longrightarrow} z^{\leftarrow},
$
as $t\toi$, in $D([0,\infty),[0,\infty))$ as well. 
\end{lemma}
\begin{proof}
Since $z$ is strictly increasing, $z^{\leftarrow}$ is continuous.  According to  Theorem 2.15 in Jacod and Shiryaev~\cite[Chapter VI]{JacodLTSP}, it sufficies to show 
$
z_t^{\leftarrow}(u)\to z^{\leftarrow}(u),
$ for all $u\geq 0$. 
One can prove this by showing that, for an arbitrary fixed $u\geq 0$,  the function
$
 y \mapsto y^{\leftarrow}(u)
$ 
 is continuous at $z$.
However, this follows at once from Proposition 2.11 in~\cite[Chapter VI]{JacodLTSP}.
Note that this proposition is actually proved using the left--continuous generalized inverse, but  under our assumptions, the proof can be easily adapted to the right--continuous case.
\end{proof}

 According to Seneta~\cite{Seneta76} (see also~\cite{MS2004,ResHTP}),  there exists a function $\widetilde{d}$ which is an asymptotic inverse of $d$, that is 
\begin{equation}\label{d_asympt_eq}
d(\widetilde{d}(t))\sim \widetilde{d}(d(t))\sim t\,,
\end{equation}
as $t\toi$. Moreover, $\widetilde{d}$ is known to be a regularly varying function with index $\alpha$. 

Denote by
\begin{equation}\label{eq:Walfa}
W_\alpha(c)=\inf\{x:S_\alpha(x)>c\}=S_\alpha^{\leftarrow}(c)\,,
\quad c\geq 0\,,
\end{equation}
the first hitting-time process of the $\alpha$-stable subordinator $(S_\alpha(t))_{t\geq 0}$.
As we shall see in the sequel (see~\eqref{thm_inf_mean_kvg_sums:proof_1}),
 Lemma~\ref{lemma_gen_inverse_cvg} together with (\ref{def_S_alpha}) and (\ref{d_asympt_eq}) implies 
\begin{equation}\label{kvg_first_hitting_time}
\left(\frac{\tau(tc)}{\widetilde{d}(t)}\right)_{c\geq 0} \cid (W_\alpha(c))_{c\geq 0}\,,
\end{equation}
in $D([0,\infty),[0,\infty))$ endowed with $J_1$ topology.
For an
$\alpha$--stable process $S_\alpha$ and fixed $c\geq 0$, the hitting-time  $W_\alpha(c)$ has the Mittag-Leffler distribution.

\section{Extremes of randomly spaced observations}
\label{Sec:Ext}

With observations $(X_n)$ and 
interarrival times $(Y_n)$ dependent, 
it can be  difficult to describe explicitly limiting
behavior of observations $X_i$
 with index $i \leq \tau(t)$.
Here, we show that this can be done using the convergence of suitably chosen point processes based on iid random vectors
 $(X_n,Y_n)$.
Convergence of this type is well understood
in the extreme value theory, see
for instance Resnick~\cite{ResEVRVPP}.

Again, we assume that $X_1\in\mathrm{MDA}(G)$ and $Y_1\sim \mathit{RegVar}(\alpha)$, $0<\alpha<1$. For $t \geq 0$ we define point process $N_t$ as
\begin{equation}\label{def_Nt_tilde}
N_t=\sum_{i\geq 1}\delta_{\left(\frac{i}{\widetilde{d}(t)},\widetilde{X}_{t,i},\widetilde{Y}_{t,i}\right)}\,,
\end{equation}
with $N_0=0$, where 
\begin{equation}\label{def_tilde_Y_ti}
\widetilde{Y}_{t,i}=\frac{Y_i}{t} \,,
\end{equation}
and
$\widetilde{X}_{t,i}$ is defined 
 by
\begin{equation}\label{def_tilde_X_ti}
\widetilde X_{t,i}=\frac{X_i-\widetilde{b}(t)}{\widetilde{a}(t)}\,,
\end{equation}
with $\widetilde{a}(t):=a(\widetilde{d}(t))$, $\widetilde{b}(t):=b(\widetilde{d}(t))$ where $a(t)$ and $b(t)$ satisfy (\ref{MDA_condition})
and $\widetilde{d}(t)$ is defined in $(\ref{d_asympt_eq})$.

The state space
of $N_t$ 
 depends on the MDA of the observations (see (\ref{MDA_condition})),
 it can be written as $[0,\infty)\times 
\mathbb{E}$ where 
\begin{equation*}
\mathbb{E}=\begin{cases}
[-\infty,\infty]\times [0,\infty] \setminus \{(-\infty,0)\}, & \ X_1\in\mathrm{MDA}(\Lambda)\\
[0,\infty]\times [0,\infty] \setminus \{(0,0)\} , & \ X_1\in\mathrm{MDA}(\Phi_\beta)
\end{cases}\,.
\end{equation*} 
Throughout we use the standard vague topology on the space of  point measures
 $M_p([0,\infty)\times \mathbb{E})$, see Resnick~\cite{ResEVRVPP}.

Recall that
\begin{equation}\label{def_mu_0}
\widetilde{d}(t)P\left((\widetilde{X}_{t,i},\widetilde{Y}_{t,i})\in \cdot\right)\stackrel{v}{\to} \mu_0(\cdot)\,, \ \mbox{ as } t\toi\,,
\end{equation}
is necessary and sufficient for
\begin{equation*}
N_t\cid N\,,
\end{equation*}
as $t\toi$, where $N$ is Poisson random measure with mean measure $\lambda\times \mu_0$, denoted by $\mathrm{PRM}(\lambda\times \mu_0)$ (see Proposition 3.21 in Resnick~\cite{ResEVRVPP} or Theorem 1.1.6 in de Haan and Ferreira~\cite{deHaan_Ferreira}).
It turns out that \eqref{def_mu_0} is immediate when one considers  the observations
 independent of interarrival times. The same holds in the important
special case when they
are exactly equal.
In that case, one can prove an interesting invariance principle concerning
 extremely  long steps of a renewal 
process, see Basrak~\cite{BB15} and references therein.

\begin{thm}\label{thm_dependent_general}
Assume that   \eqref{def_mu_0} holds,
then
\begin{equation}\label{thm_dep_general:claim}
\left(N_t,
\frac{T(\widetilde{d}(t)\cdot)}{t},
\frac{\tau(t)}{\widetilde{d}(t)}\right)
\cid (N,S_\alpha(\cdot), W_\alpha ) \,,
\end{equation}
 as $t\toi$,
in the product space $M_p \times D[0,\infty)  \times  \mathbb{R} $ and the corresponding  product topology (of  vague,  $J_1$ and Euclidean topologies),
where 
$N$ is a $\mathrm{PRM}(\lambda\times\mu_0)$,
while $S_\alpha$ and 
 $W_\alpha=W_\alpha(1)$ 
denote the $\alpha$--stable subordinator and the first passage time from \eqref{def_S_alpha} and \eqref{kvg_first_hitting_time} respectively.
\end{thm}

\begin{remark} \label{rem:prva}
If one denotes the limiting point process above as
$N= \sum_i \delta_{(T_i,P_i,Q_i)}$, then $S_\alpha$ has representation 
\[
 S_\alpha(t) = \sum_{T_i\leq  t } Q_i\,,
\]
while $W_\alpha(c)$ is the inverse of the increasing process
$S_\alpha$ as in \eqref{eq:Walfa}.
Actually, due to lemma~\ref{lemma_gen_inverse_cvg},
we prove a stronger result than the theorem above claims.
Namely, under the same assumption, \eqref{thm_dep_general:claim}
holds jointly with
\eqref{kvg_first_hitting_time}.
\end{remark}

\begin{proof}
The first part of the proof is standard, and essentially follows the lines of the proof of Theorem 7.1 in Resnick~\cite{ResHTP}.
Observe that for $s,t \geq 0$ and $T(\cdot)$ in (\ref{def_T(t)})
\[
\frac{T(\widetilde{d}(t)s)}{t}=\sum_{i=1}^{\lfloor \widetilde{d}(t)s\rfloor}\widetilde{Y}_{t,i}\,.
\]
From \eqref{def_S_alpha} we know
\begin{equation}\label{thm_inf_mean_kvg_sums}
\frac{T(\widetilde{d}(t)\cdot)}{t}\cid S_\alpha(\cdot)\,, \text{ as } t\toi\,,
\end{equation}
in $D([0,\infty),[0,\infty))$ with $J_1$ topology. The limiting process $S_\alpha(\cdot)$
has the same distribution as in (\ref{def_S_alpha}). 
Notice that the generalized inverse of $T(\widetilde{d}(t)\cdot)/t$ equals
\begin{eqnarray*}
\left(\frac{T(\widetilde{d}(t)\cdot)}{t}\right)^{\leftarrow}(u)&=&\inf\left\{s:\sum_{i=1}^{\lfloor \widetilde{d}(t)s \rfloor}Y_i>tu\right\}= 
\mbox{substituting } k=\lfloor \widetilde{d}(t)s \rfloor\\
&=&\frac{\inf\left\{k\in\mathbb{N}:\sum_{i=1}^{k}Y_i>tu\right\}}{\widetilde{d}(t)}=\frac{\tau(tu)}{\widetilde{d}(t)}\,,
\end{eqnarray*}
for every $u\geq 0$, where $\tau(\cdot)$ is defined in (\ref{def_renewal_process}).
Using (\ref{thm_inf_mean_kvg_sums}) and Lemma~\ref{lemma_gen_inverse_cvg}, we apply continuous mapping theorem to obtain
\begin{equation}\label{thm_inf_mean_kvg_sums:proof_1}
\frac{\tau(t \cdot)}{\widetilde{d}(t)}\cid W_\alpha(\cdot)\,,
\end{equation}
in $D([0,\infty),[0,\infty))$ with $J_1$ topology, where $W_\alpha\eind S_\alpha^{\leftarrow}$.
Finally, the continuous mapping argument yields the joint convergence
in \eqref{thm_dep_general:claim}.

\end{proof}

Abusing the  notation somewhat, for any time period $A\subseteq [0,\infty)$
and an arbitrary  point measure $n \in M_p([0,\infty)\times\mathbb{E})$, we introduce restricted point measure
\begin{equation}\label{e:nRestr}
n\Big|_{A} 
\ \mbox{ as }\ 
n\Big|_{A\times\mathbb{E}}\,.
\end{equation}
Since the distribution of point processes $N_t$ contains the information about all upper order statistics in the sequence $(X_n)$,
it is useful to study the limit of point processes $N_t$  restricted to time intervals determined by the renewal process $(\tau(t))$.
For an illustration, denote by $M^{\tau}(t),\  t\geq 0$,  the maximum of observations $\{X_1,\ldots,X_{\tau(t)}\}$.
Note that
\[ 
\frac{M^\tau(t)-\widetilde{b}(t)}{\widetilde{a}(t)}\leq x \quad 
\mbox{ if and only if }\quad 
N_t \Big|_{\left[0,\frac{\tau(t)}{\widetilde{d}(t)}\right]\times(x,\infty]\times [0,\infty] } = 0 
\,, \]
for all suitably chosen $x$.
Using similar arguments, the following theorem yields  the limiting distribution of all extremes
until a given time $t$.

\begin{thm}\label{thm_restriction}
Assume that   \eqref{def_mu_0} holds,
then
\begin{equation}\label{e:NtRestr}
N_t \Big|_{\left[0,\frac{\tau(t)}{\widetilde{d}(t)}\right]}
\cid 
N \Big|_{\left[0,W_\alpha \right]}
\quad
\mbox{ and   }\quad
N_t \Big|_{\left[0,\frac{\tau(t)}{\widetilde{d}(t)}\right)}
\cid 
N \Big|_{\left[0,W_\alpha \right)}\,,
\end{equation}
 as $t\toi$,
where 
$N$ is a $\mathrm{PRM}(\lambda\times\mu_0)$,
and 
 $W_\alpha=W_\alpha(1)$ 
denotes the fist passage time from \eqref{kvg_first_hitting_time}.
\end{thm}

\begin{proof}
 Observe that, the limiting point  process $N$ 
in general has a point exactly at the right end point of the 
interval $\left[0,W_\alpha \right]$. This forbids direct application
of the continuous mapping argument. Fortunately,
one can adapt the argument from the proof of Theorem 4.1
in \cite{BB15}.

Denote by $n, n_t, \  t>0$, arbitrary Radon
point measures in $M_p([0,\infty)\times \mathbb{E})$.
One  can always write
$$
 n_t = \sum_i \delta_{(v^t_i,x^t_i,y^t_i)},\  n = \sum_i \delta_{(v_i,x_i,y_i)}
$$
for some sequences $(v^t_i,x^t_i,y^t_i)$ and  $(v_i,x_i,y_i)$ 
with values in $[0,\infty)\times \mathbb{E}$.
Denote further their corresponding cumulative sum functions, by
$$
 s_t(u) = \dint_{[0,u]\times\mathbb{E}} y\, n_t(dv,dx,dy)\,,\ u \geq 0 \,,
$$
and
$$
 s(u) = \dint_{[0,u]\times\mathbb{E}} y\, n(dv,dx,dy)\,,\ u \geq 0 \,,
$$
Suppose that they take finite value for each $u>0$, but tend to infinity as $u\toi$. This makes 
$s_t$ and $s$ well defined, unbounded,
nondecreasing elements of the space of c\`adl\`ag functions $D[0,\infty)$.
Their right--continuous generalized inverses (or hitting time functions) we denote
by $s^{\leftarrow}$ and $s_t^{\leftarrow}$. 

Assume that  
\begin{equation} \label{e:convof2}
 (n_t, s_t)
  \to (n, s)\,,
\end{equation}
in the product topology (of vague and $J_1$  topologies) as $t\toi$.
If
\begin{equation} \label{e:condons}
 0<s\Big(s^{\leftarrow}(1)-\Big)<1 <s\Big(s^{\leftarrow}(1)\Big)  \mbox{ and }  n (\{v\}\times (0,\infty)\times\mathbb{E})  \leq 1
\end{equation}
for all $v \geq 0$, by straightforward
adaptation of  Theorem 4.1 \cite{BB15}  to three--dimensional case
 \begin{equation}\label{NRestric}
  n_t\big|_{[0,s_t^{\leftarrow}(1)]}  \civ 
  n\big|_{[0,s^{\leftarrow}(1)]} \,
\quad
\mbox{ and }  \quad
  n_t\big|_{[0,s_t^{\leftarrow}(1))}  \civ 
  n\big|_{[0,s^{\leftarrow}(1))} \,.
\end{equation}
We can apply continuous mapping theorem now. Observe now that by Theorem~\ref{thm_dependent_general}
\[
\left(N_t,
\frac{T(\widetilde{d}(t)\cdot)}{t}\right)
\cid (N,S_\alpha(\cdot) )\,.
\]
On the other hand, the limiting Poisson process $N$ and  $\alpha$--stable subordinator  $S_\alpha$ satisfy
regularity assumptions \eqref{e:condons} with probability one, and therefore \eqref{e:NtRestr} holds.

\end{proof}

The
limiting distribution of all upper order statistics until a given time $t$, determined by the theorem above, can be quite complicated depending on
 the joint distribution of $N$ and $W_\alpha$ in $(\ref{thm_dep_general:claim})$. Hence, we examine two particular types of dependence between  $X_n'$s and $Y_n'$s in detail. The first of them is called \textit{the asymptotic tail independence} (see e.g.  de Haan and Ferreira~\cite{deHaan_Ferreira} and Resnick~\cite{ResEVRVPP}).
Rougly speaking, it requires that when $Y_n$ is large, there is negligible probability of $X_n$  
being large. The second type of dependence
we consider in detail is called \textit{the asymptotic full tail dependence}.
Intuitively, it implies that the $X_n$ and $Y_n$ are highly tail dependent in the sense that if one of them is large, then the other one  is also large, asymptotically with probability $1$ (see Sibuya~\cite{Sibuya60}, de Haan and Resnick~\cite{deHaan_Res77} or Resnick~\cite[pp.296--298]{ResEVRVPP}).

\subsection{Asymptotic tail independence}\label{subsection_asym_tail_indep}

Recall that  $((X_n,Y_n))$ is an iid sequence of random vectors such that $X_1\in\mathrm{MDA}(G_1)$ where $G_1=\Lambda$ or $\Phi_\beta$ with $\beta>0$ and $Y_1\sim\mathit{RegVar}(\alpha)$ for $\alpha\in(0,1)$. Hence, $Y_1\in\mathrm{MDA}(G_2)$ where $G_2=\Phi_\alpha$. 
By $F_{X,Y}$ denote the joint cdf of $(X_1,Y_1)$, and set $U_X=1/(1-F_X)$, $U_Y=1/(1-F_Y)$. Note, $U_X(X)$ and $U_Y(Y)$ are $\mathit{RegVar}(1)$ at infinity.

It is known, cf. de Haan and Resnick~\cite{deHaan_Res77}, that 
for the so-called tail independent $X_1$ and $Y_1$, the measure $\mu_0$ in (\ref{def_mu_0}) is concentrated on the axes. In particular, for $X_1\in\mathrm{MDA}(\Lambda)$ and $(x,y)\in[-\infty,\infty]\times[0,\infty]\backslash \{(-\infty,0)\}$, we have 
\begin{equation}\label{def_mu_0_Gubmel}
\mu_0\Big(([-\infty,x]\times[0,y])^c\Big)=-\log{G_1(x)}-\log{G_2(y)}=e^{-x}+y^{-\alpha} \,.
\end{equation}
If $X_1\in\mathrm{MDA}(\Phi_\beta)$ and $(x,y)\in[0,\infty]^2\backslash\{(0,0)\}$, then
\begin{equation}\label{def_mu_0_Frechet}
\mu_0\Big(([0,x]\times[0,y])^c\Big)=-\log{G_1(x)}-\log{G_2(y)}= x^{-\beta}+y^{-\alpha} \,.
\end{equation}
Let $\{N_t:t \geq 0\}$ be point processes from $(\ref{def_Nt_tilde})$. It is known that 
$ N_t \cid N \,, $
where $N$ is $\mathrm{PRM}(\lambda\times \mu_0)$ is equivalent to  $F_{X,Y}\in\mathrm{MDA}(G)$ (we refer to Resnick~\cite[Section 5.4]{ResEVRVPP} for a definition of multivariate $\mathrm{MDA}$) with
\begin{equation}\label{def_G}
G(x,y)=\begin{cases}
\exp\left\{-\mu_0\Big(([-\infty,x]\times[0,y])^c\Big)\right\}, & \ X_1\in\mathrm{MDA}(\Lambda)\\
\exp\left\{-\mu_0\Big(([0,x]\times[0,y])^c\Big)\right\}, & \ X_1\in\mathrm{MDA}(\Phi_\beta)
\end{cases}\,,
\end{equation}
and $\mu_0$ defined in $(\ref{def_mu_0_Gubmel})$ and $(\ref{def_mu_0_Frechet})$.
The measure $\mu_0$ is often called the exponent measure.

Under our assumptions, necessary and sufficient condition for \eqref{def_mu_0} can be inferred from the literature (cf. Theorem 6.2.3 in de Haan and Ferreira~\cite{deHaan_Ferreira}). Next theorem summarizes this for completeness. 

\begin{thm}\label{thm_asymptotic_independence}
For measure $\mu_0$ described in (\ref{def_mu_0_Gubmel}) and (\ref{def_mu_0_Frechet}), \eqref{def_mu_0} is equivalent to
\begin{equation}\label{cond_asymptotic_indep} 
\lim_{x\toi}P(X_1>U_X^{\leftarrow}(x)|Y_1>U_Y^{\leftarrow}(x))=0 \,.
\end{equation}
\end{thm}

\begin{remark} \label{rem:Ind}
If \eqref{cond_asymptotic_indep} holds, $X_1$ and $Y_1$ are called asymptotically tail independent. 
Under this condition, the limiting Poisson process $N$ in \eqref{thm_dep_general:claim}
can be decomposed into two independent parts, e.g.
for $X_1\in\mathrm{MDA}(\Lambda)$ 
\[
 N= \sum_i \delta_{(T_i,P_i,0)} + \sum_i \delta_{(T'_i,0,Q_i)}\,.
\]
This makes the restriction of $N$ to
the first two coordinates
$$
N^{(2)} = \sum_{i\geq 1} \delta_{(T_i,P_i)}.
$$
independent of $S_\alpha   =\sum_{T'_i\leq  t } Q_i$, and  $W_\alpha$ 
for the same reason.
\end{remark}

\begin{proof}
Under our assumptions, \eqref{def_mu_0} is known to be equivalent to $F_{X,Y}\in\mathrm{MDA}(G)$ for $G$ given by \eqref{def_G}.
Since $U_X(X)$ and $U_Y(Y)$ are $\mathit{RegVar}(1)$ at infinity, by Sibuya's theorem (see Theorem 5 in de Haan and Resnick~\cite{deHaan_Res77}), $F_{X,Y}\in\mathrm{MDA}(G)$ is further equivalent to
\begin{equation}\label{thm_asympt_indep:cond_proof}
\lim_{x\toi}P(U_X(X_1)>x|U_Y(Y_1)>x)=0 \,.
\end{equation}
It remains to prove the equivalence between \eqref{cond_asymptotic_indep} and \eqref{thm_asympt_indep:cond_proof}.
From Theorem 1.7.13 in Leadbetter et al.~\cite{Leadbetter83} we conclude that
\[
\frac{P( X= U_X^{\leftarrow}(x))}{P( X\geq U_X^{\leftarrow}(x))} \to 0\,,
\]
as $x\toi$.
By the proof of Proposition 5.15 in Resnick~\cite{ResEVRVPP}, on the other hand,
$
x P( X >  U_X^{\leftarrow}(x)) \to 1\,,
$
as $x\toi$.
Now, observing that
\[
\left\{ U_X(X) > x \right\} \subseteq \left\{ X \geq U_X^{\leftarrow}(x)\right\}\,
\mbox{ and }
\left\{ X > U_X^{\leftarrow}(x)\right\}  \subseteq \left\{ U_X(X) > x \right\}\,. 
\]
one can show that   \eqref{thm_asympt_indep:cond_proof} is equivalent to \eqref{cond_asymptotic_indep}.
\end{proof}

Condition $(\ref{cond_asymptotic_indep})$ does not  seem easy to verify directly in general.
In some examples one can verify a simpler sufficient condition
introduced in the following lemma (cf. condition A introduced in Anderson~\cite{Anderson87}).

\begin{lemma}\label{prop_weak_Anderson}
If
\begin{equation}\label{cond_weak_Anderson}
\lim_{x\toi}\limsup_{y\toi}P(X_1>x|Y_1>y)=0 \,,
\end{equation}
then
\begin{equation*}
\lim_{x\toi}P(X_1>U_X^{\leftarrow}(x)|Y_1>U_Y^{\leftarrow}(x))=0 \,.
\end{equation*}
\end{lemma}
\begin{proof}
Assuming $(\ref{cond_weak_Anderson})$, for every $\varepsilon>0$ there exists $x_1\in\mathbb{R}$ such that for all $x\geq x_1$ we have
\[ \limsup_{y\toi}P(X_1>x|Y_1>y)<\varepsilon\,. \]
In particular,
\[ \limsup_{y\toi}P(X_1>x_1|Y_1>y)<\varepsilon \,. \]
Moreover, there exists $y_0>0$ such that for all $y>y_0$ we have
\[ P(X_1>x_1|Y_1>y_0)<\varepsilon \,.\]
If we denote $x_0=\inf\{x>0:U_X^{\leftarrow}(x)>x_1,U_Y^{\leftarrow}(x)>y_0\}$, then for all $x>x_0$ 
\[ P(X_1>U_X^{\leftarrow}(x)|Y_1>U_Y^{\leftarrow}(x))\leq P(X>x_1|Y_1>U_Y^{\leftarrow}(x))<\varepsilon \,.\]
\end{proof}

An application of Theorem~\ref{thm_asymptotic_independence}, yields the asymptotic behavior of the $k$-th upper order statistics in a sample indexed by the renewal process $(\tau(t))$.

\begin{example} \label{ex:Ind} Let $M_k^{\tau}(t),\  t\geq 0$ represent the $k$-th upper order statistics in the sample  $\{X_1\ldots,X_{\tau(t)}\}$. Under the assumptions of Theorem~\ref{thm_asymptotic_independence} we can find the limiting distribution for suitable normalized random variables 
$M_k^\tau(t)$.
Clearly, from \eqref{e:NtRestr} we obtain
\begin{eqnarray*}
 \lefteqn{P\left(M_k^\tau(t)\leq \widetilde{a}(t)x+\widetilde{b}(t)\right)=P\left(N_t\left([0,\tau(t)/\widetilde{d}(t)]\times(x,\infty]\times \mathbb{R}_+\right)\leq k-1\right)}\\
&\to& P\Big(N\big([0,W_\alpha]\times (x,\infty]\times \mathbb{R}_+\big)\leq k-1\Big)
=E\left(\frac{\Gamma_k \left(W_\alpha\, \mu_G ((x,\infty]\times \mathbb{R}_+)\right)}{\Gamma(k)}\right)\,, 
\end{eqnarray*}
as $t\toi$, where $H_\alpha(x)=P(W_\alpha \leq x)$ represents the cdf of the random variable $W_\alpha$ and $\Gamma_k(x)$ is an incomplete gamma function. For $k=1$, i.e. for the partial maxima of the first $\tau(t)$ observations, the result first appears in Berman~\cite{Berman62}. For linearly growing $\tau(t)$
independent of the observations, a similar result can be found in Theorem 4.3.2 of
 Embrechts et al.~\cite{EKMikosch}.
\end{example}

\subsection{Asymptotic full tail dependence}\label{subsection_asym_full_dep}

In the case when observations and interarrival times are exactly equal, the limiting
behavior of the maximum has been found already by Lamperti~\cite{Lamperti61}.
We will show here that one can extend his results to study all the upper order 
statistics in a more general setting.
We keep the assumptions and notation  from subsection~\ref{subsection_asym_tail_indep}.
The main difference is that the limiting measure $\mu_0$  in (\ref{def_mu_0}) will be concentrated on 
a line, i.e. on the set
\begin{equation*}
C=\begin{cases}
\{(u,v)\in(-\infty,\infty)\times(0,\infty):e^{-u}=v^{-\alpha}\}, & \text{if } G_1=\Lambda\\
\{(u,v)\in(0,\infty)\times(0,\infty):u^{-\beta}=v^{-\alpha}\}, & \text{if } G_1=\Phi_\beta \,.
\end{cases}
\end{equation*}
More precisely, for $y>0$ and $C_{(y,\infty)}=\{(u,v)\in C: v>y\}$, the measure $\mu_0$ is determined by
\begin{equation*}
\mu_0(C_{(y,\infty)})=y^{-\alpha} \,.
\end{equation*}

Under these assumptions, necessary and sufficient condition for  (\ref{def_mu_0}) 
is the  full tail dependence condition known from the literature (cf. de Haan and Resnick~\cite{deHaan_Res77}). We again summarize this 
 for completeness. 

\begin{thm}\label{thm_asympt_full_dependence} 
For measure $\mu_0$ given above,
 \eqref{def_mu_0} is equivalent to
\begin{equation}\label{cond_asympt_full_dep} 
\lim_{x\toi}P(X > U_X^{\leftarrow}(x)|Y > U_Y^{\leftarrow}(x))=1 \,.
\end{equation}
\end{thm}

\begin{proof}
The proof follows the lines of the proof of Theorem~\ref{thm_asymptotic_independence}, except that, instead of Theorem 5, we use Theorem 6 in de Haan and Resnick~\cite{deHaan_Res77}).
\end{proof}

For an application of Theorem~\ref{thm_asympt_full_dependence} we refer to section~\ref{section_excursions}, where the problem of the longest excursion of the continuous time random walk is considered.

\section{Excursions and sojourn times of continuous time random walk}
\label{Sec:CTRW}

In this section we use results from section 3 to obtain the limiting distribution of extremely long sojourn times at level zero of a CTRW. We consider a CTRW which is a simple symmetric random walk subordinated to a certain renewal process. Hence, jumps are always of magnitude one, while waiting times between jumps correspond to interarrival times of the subordinating renewal process.

Let $(E_n)_{n\geq 1}$ be an iid sequence of non-negative random variables with finite expectation. Suppose $E_1\in\mathrm{MDA}(G)$ where $G=\Lambda$ or $G=(\Phi_{\beta})$, for $\beta > 1$. Denote the partial sum of the sequence $(E_n)$ by $T(n)=\sum_{i=1}^{n}E_i$ and set $T(0)=0$. Additionally, let $(N(t))_{t\geq 0}$ denote a renewal process generated by the sequence $(E_n)$, that is
\[ N(t)=\max\left\{k\geq 0: T(k)=\sum_{i=1}^{k}E_i\leq t\right\} \,.\]
The sequence $(E_n)$ models the waiting times between jumps, whereas the renewal process $(N(t))$ counts the number of jumps up to a time $t$. Furthermore, let $(\varepsilon_n)_{n\geq 1}$ be the iid sequence of Rademacher random variables, that is $P(\varepsilon_1=1)=P(\varepsilon_1=-1)=1/2$.  Let $S_n=\sum_{i=1}^{n}\varepsilon_i$ denote the partial sums of $(\varepsilon_n)$. Set $S_0=0$ and define continuous time random walk
as the process 
\[ Z(t)=S_{N(t)}=\sum_{i=1}^{N(t)}\varepsilon_i \,,\,\quad  t\geq 0 \,.\] 
Our goal is to determine the limiting distribution of the longest time interval during which CTRW $Z(\cdot)$ remains at level zero before time $t$, including the last 
possibly incomplete part. With that in mind, we define some auxiliary random variables.

\subsection{The longest sojourn time at level zero}\label{section_sojourn_times}

Let $X_1$ denote the time spent during the first visit to the origin. Clearly, $X_1=E_1$. Since $Z(t)=0$ for all $t\in[0,E_1)$, and $Z(E_1)\neq 0$, the time of the first return to the origin is defined as
\[ A_1=\inf\{t \geq X_1 : Z(t)=0\} \,.\]
Further, if we set $Y_1=A_1$, then $Y_1$ represents the sum of the first sojourn time
 at the origin and the time spent in the first excursion away from the origin. Clearly, duration of the first excursion, denoted by $R_1$, satisfies $R_1=Y_1-X_1$.
If we set $A_0=0$, then for $i\geq 1$, the time spent during the $i$-th visit at the origin or the time spent on the $i$-th excursion can be defined recursively as
\begin{eqnarray*}
X_i&=&\inf\{t\geq A_{i-1}: Z(t)\neq 0\}-A_{i-1}=E_{N(A_{i-1})+1}\,,\\
A_i&=&\inf\{t\geq A_{i-1}+X_i:Z(t)=0\}\,,\\
Y_i&=&A_i-A_{i-1}\,,\\
R_i&=&Y_i-X_i\,.
\end{eqnarray*}
An illustration of these random variables is given in Figure~\ref{CTRW_figure}.
\begin{figure}[h]\label{CTRW_figure}
\begin{center}
 \includegraphics[scale=0.7]{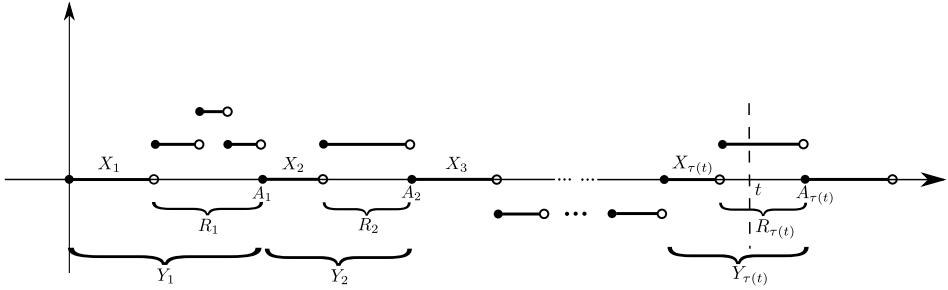}
\end{center}
\caption{A sketch of CTRW and associated random variables $X_i$, $A_i$, $Y_i$, $R_i$.}
\end{figure}

Clearly, the  sequence $(X_n)_{n\geq 1}$ defined above  is iid. Moreover, $X_1$ belongs to the same MDA as $E_1$. Since, $(Y_n)_{n\geq 1}$ is iid sequence of non-negative random variables, the renewal process
\begin{equation}\label{Exmpl_CTRW_def_tau}
 \tau(t)=\inf\left\{ k\geq 1:\sum_{i=1}^k Y_i > t \right\} \ , \, t\geq 0\,,
\end{equation}
is well defined.

To apply Theorem~\ref{thm_restriction} we need to determine the joint tail behavior of random variables $Y_i$ and $R_i$. Note that $(R_n)_{n\geq 1}$ is an iid sequence of non-negative random variables, and for $u>0$, it satisfies
\[ P(R_1>u)=P\left(\sum_{i=2}^{K}E_i>u\right) \,,\]
where $K=\inf\{m\geq 1: S_m=0\}$, and $S_n=\sum_{i=1}^{n}\varepsilon_i$.
 It is well known (see e.g. Durrett~\cite[Section 4.3]{Durrett2010}) that $K\sim \mathit{RegVar}(1/2)$. Therefore, for instance, by Proposition 4.3. in Fa\"y et al.~\cite{Genna06} 
\[ P(R_1>u)\sim E(\sqrt{E_1}) P(K>u) \,, \]
as $u\toi$. In particular, $R_1\sim \mathit{RegVar}(1/2)$, and, consequently, $Y_1\sim \mathit{RegVar}(1/2)$.

To verify  condition \eqref{cond_weak_Anderson}, observe that $Y_1 = R_1 +X_1 $ with $R_1$ and $X_1$ independent. Therefore
\[
\frac{P(Y_1>y,X_1>x)}{P(Y_1>y)} \leq
\frac{P(R_1>y - y ^{2/3})\, P(X_1>x)}{P(Y_1>y)}+
\frac{P(X_1> y ^{2/3})}{P(Y_1>y)}\,.
\]
The second term on the right tends to 0 as $y\toi$ because $X_1$ has finite mean and
$Y_1\sim \mathit{RegVar}(1/2)$,
while the first term is bounded by $P(X_1>x)$, which clearly tends to 0, if we let first $y$,
then $x\toi$. 

If we construct point processes $N_t$, $t\geq 0$ as in (\ref{def_Nt_tilde}), observe that functions $d$, $\widetilde{d}$ are regularly varying with indices $2$ and $1/2$, respectively.
Now, using Theorems~\ref{thm_dependent_general} and~\ref{thm_asymptotic_independence} together with Lemma~\ref{prop_weak_Anderson}, we obtain 
\begin{equation}\label{Exmpl_CTRW_lng_stay_joint_cvg}
 \left(N_t,\frac{\tau(t)}{\widetilde{d}(t)}\right)\cid \left(N,W_{1/2}\right)\,,
\end{equation}
as $t\toi$, where $N\sim \mathrm{PRM}(\lambda\times \mu_0)$ 
is such that 
$N^{(2)}$ (see Remark~\ref{rem:Ind}) is
independent of the random variable $W_{1/2}$.

The following  theorem describes the asymptotic distribution of the longest sojourn time of CTRW at level zero. Denote the longest sojourn time at level 0 up to time $t$ by $Q(t)$. First observe that, if
\[ \sum_{i=1}^{\tau(t)-1}Y_i+X_{\tau(t)}<t \,,\] 
$Q(t)$ simply equals  $M^\tau(t)=\max\{X_1,\ldots , X_{\tau(t)} \}$.
On the other hand, if
\[ t \leq \sum_{i=1}^{\tau(t)-1}Y_i+X_{\tau(t)} \,,\] 
\[ 
Q(t) = \max\left\{M^{\tau-1}(t),t-\sum_{i=1}^{\tau(t)-1}Y_i\right\}\,.
\]
In either case 
\begin{equation}\label{eq:MQM}
 M^{\tau -1} (t) \leq  Q(t) \leq  M^{\tau } (t)\,.
\end{equation}
Observe further that the random events
\[
\left\{\frac{M^\tau(t)-\widetilde{b}(t)}{\widetilde{a}(t)} \leq x \right\}\,
\ \mbox{ and } \ 
\left\{\frac{M^{\tau-1}(t)-\widetilde{b}(t)}{\widetilde{a}(t)} \leq x \right\}
\]
correspond to
\[
\left\{N_t\Big|_{\left[0,\frac{\tau(t)}{\widetilde{d}(t)}\right]\times(x,\infty]\times[0,\infty]}=0\right\}
\ \mbox{ and } \ 
\left\{N_t\Big|_{\left[0,\frac{\tau(t)}{\widetilde{d}(t)}\right)\times(x,\infty]\times[0,\infty]}=0\right\}\,.
\]

\begin{thm}
Under the assumptions  above
\begin{equation*}
P\left(\frac{Q(t)-\widetilde{b}(t)}{\widetilde{a}(t)}\leq x\right)\to E\left(G(x)^{W_{1/2}}\right)\,,
\end{equation*}
as $t\toi$.
\end{thm}
\begin{proof}
Using (\ref{Exmpl_CTRW_lng_stay_joint_cvg}) and Theorem~\ref{thm_restriction}, 
by the discussion before the theorem for an arbitrary $x\in \mathbb{R}$,  we obtain
\[
P\left(\frac{M^\tau(t)-\widetilde{b}(t)}{\widetilde{a}(t)}\leq x\right)
\to P\left(N\Big\vert_{\left[0,W_{1/2}\right]\times (x,\infty]\times[0,\infty]}=0\right)\,, \]
as $t\toi$, and similarly
\[
 P\left(\frac{M^{\tau-1}(t)-\widetilde{b}(t)}{\widetilde{a}(t)}\leq x\right)
 \to P\left(N\Big\vert_{\left[0,W_{1/2}\right)\times (x,\infty]\times[0,\infty]}=0\right)\,.
 \]
By the independence between $N^{(2)}$ and $W_{1/2}$, see Remark~\ref{rem:Ind}, both limiting probabilities above equal
\begin{eqnarray*}
\lefteqn{P\left(N^{(2)}\Big\vert_{\left[0,W_{1/2}\right]\times (x,\infty]}=0\right)}\\
&=&\int_{0}^{+\infty}e^{-s\mu_0\left(\big([-\infty,x]\times[0,\infty]\big)^c\right)}dF_{W_{1/2}}(s)=\int_{0}^{+\infty}G(x)^{s}dF_{W_{1/2}}(s)\\
&=&E\left(G(x)^{W_{1/2}}\right)\,.
\end{eqnarray*}
By \eqref{eq:MQM}
\[
P\left(\frac{Q(t)-\widetilde{b}(t)}{\widetilde{a}(t)}\leq x\right)
\to E\left(G(x)^{W_{1/2}}\right)\,,
\]
as $t\toi$.
\end{proof}

Similarly, one can use (\ref{Exmpl_CTRW_lng_stay_joint_cvg}) to obtain the joint distribution of the two longest sojourn times of CTRW. Denote by $Q(t)$ and by $Q'(t)$  the longest and the second longest sojourn times until time $t$. Fix levels $u_1>u_2>0$. Observe that
\begin{eqnarray*}
\lefteqn{P\left(\frac{Q(t)-\widetilde{b}(t)}{\widetilde{a}(t)}\leq u_1,\frac{Q'(t)-\widetilde{b}(t)}{\widetilde{a}(t)}\leq u_2\right)}\\
&=&P\Bigg(N_t\Big([0,\tfrac{\tau(t)}{\widetilde{d}(t)}]\times(u_1,\infty]\times[0,\infty]\Big)=0,N_t\Big([0,\tfrac{\tau(t)}{\widetilde{d}(t)}]\times(u_2,\infty]\times[0,\infty]\Big)\leq 1\Bigg) \,.
\end{eqnarray*}
Now one can apply 
(\ref{Exmpl_CTRW_lng_stay_joint_cvg}) and Theorem~\ref{thm_restriction} to show that the last
expression converges to 
\[
 E\left(G(u_2)^{W_{1/2}}\right)+\mu_G(u_2,u_1]E\left(W_{1/2}G(u_2)^{W_{1/2}}\right)\,. 
\]

\subsection{The longest excursion}\label{section_excursions}

Our next goal is to determine the limiting distribution of the length of the longest excursion completed until time $t$, i.e. the longest time interval during which CTRW $Z(t)$ is not equal to zero, completed until time $t$. 
In contrast to subsection~\ref{section_sojourn_times}, here we only assume  that $E_i$'s 
have finite expectation.
As in subsection~\ref{section_sojourn_times}, we denote by $R_i$ the time spent on the $i$-th excursion, and by $Y_i$ the total time spent on the $i$-th stay at zero and the $i$-th excursion. Now, we are interested in determining the limiting distribution of 
\begin{equation}\label{Exmpl_lon_excursion_def_M_tau}
M^{\tau-1}(t)=\sup\{R_i:i \leq \tau(t)-1\}\,,
\end{equation}
with $\tau(t)$  given in (\ref{Exmpl_CTRW_def_tau}), which
 corresponds to the length of the longest completed excursion.

In the present model, the point processes ${N}_t$ of Theorem~\ref{thm_asympt_full_dependence} are
constructed using
 sequences $(R_n)$ (instead of $(X_n)$) and $(Y_n)$. Recall that $U_R=1/(1-F_R)$, and
 observe that $ U_R^{\leftarrow}(x) \leq U_Y^{\leftarrow}(x)$ for all $x$, since
 $Y_1=R_1+X_1$. 
 To show \eqref{cond_asympt_full_dep},
 observe that
\[
 x P( Y_1>U_Y^{\leftarrow}(x)) \to 1\,,
\] 
and, since $Y_1$ is $\mathit{RegVar}(1/2)$, it is well know that the same holds for
$R_1=Y_1-X_1$ because $EX_1<\infty$. Moreover, one can show that
\[
x P( R_1 >U_Y^{\leftarrow}(x)) =  x P( Y_1-X_1 >U_Y^{\leftarrow}(x)) \to 1\,.
\]
Clearly,
 \begin{eqnarray*}
 \lefteqn{ P(R_1>U_R^{\leftarrow}(x)\,|\, Y_1>U_Y^{\leftarrow}(x))  }\\
 & = &
 \frac{P(R_1>U_R^{\leftarrow}(x)\,, R_1+X_1>U_Y^{\leftarrow}(x)) }{P(Y_1>U_Y^{\leftarrow}(x))}
 \geq \frac{P(R_1>U_Y^{\leftarrow}(x)) }{P(Y_1>U_Y^{\leftarrow}(x))}\,.
 \end{eqnarray*}
Since the last ratio above tends to 1 as $x\toi$, \eqref{cond_asympt_full_dep} is proved.

From Theorems~\ref{thm_restriction} and \ref{thm_asympt_full_dependence}, we obtain
\begin{equation}\label{Exmpl_lon_excursion_joint_cvg}
N_t \Big|_{\left[0,\frac{\tau(t)}{\widetilde{d}(t)}\right)}
\cid 
N \Big|_{\left[0,W_{1/2} \right)}
\end{equation}
as $t\toi$. In this case,
\[
{N} = \sum_{i\geq 1} \delta_{(T_i,P_i,P_i)}\,,
\]
while 
$
W_{1/2}= W_{1/2}(1) = S_{1/2}^{\leftarrow}(1)
$
with $S_{1/2}$ equal to the $1/2$--stable subordinator
\[
S_{1/2}(t) = \sum_{T_i\leq t} P_i \,,
\]
which makes $W_{1/2}$ completely  dependent on
$
{N}^{(2)} =   \sum_{i\geq 1} \delta_{(T_i,P_i)}\,,
$
cf. Remark~\ref{rem:Ind}.

Observe that (\ref{Exmpl_lon_excursion_joint_cvg}) identifies the limiting distribution of all upper order statistics in the sequence of excursions. It follows that rescaled excursions behave asymptotically as 
completed jumps of the subordinator  $S_{1/2}$ until the passage of
the level $1$. Distribution of those jumps is well understood, 
see  Perman~\cite{Perman93} or
Pitman and Yor~\cite{PitmanYor97}. 
Alternatively, one could include the last, possibly incomplete excursion in the
analysis. It turns out that the convergence result still holds,
but one needs to add the last a.s. incomplete jump of the subordinator $S_{1/2}$
to the limiting distribution. All those jumps have the distribution of the Pitman--Yor
point processes, see \cite{Perman93,PitmanYor97}.


To illustrate our claims, the next theorem gives the limiting distribution of $M^{\tau-1}(t)$.

\begin{thm} \label{Thm:Exc}
Under the assumptions above, $M^{\tau-1}(t)$  in (\ref{Exmpl_lon_excursion_def_M_tau})
satisfies
\begin{equation*}
\frac{M^{\tau-1}(t)}{t}\cid V\,,
\end{equation*}
as $t\toi$, where $V$ has the distribution of the largest jump of $1/2$--stable suboordinator
completed before the passage of the level 1.
\end{thm}
 
\begin{proof}
Observe that for arbitrary $x>0$
\[ P\left(\frac{M^{\tau-1}(t)}{\widetilde{a}(t)}\leq x\right)=P\left({N}_t\Big|_{\left[0,\frac{\tau(t)}{\widetilde{d}(t)}\right)\times (x,\infty]\times[0,\infty]}=0\right) \, \]
Theorem~\ref{thm_restriction} yields  \eqref{Exmpl_lon_excursion_joint_cvg}, and therefore
\[ P\left({N}_t\Big|_{\left[0,\frac{\tau(t)}{\widetilde{d}(t)}\right)\times  (x,\infty]\times[0,\infty]}=0\right)\to P\left({N}\Big|_{[0,W_{1/2})\times (x,\infty]\times[0,\infty]}=0\right) \,. \]
According to the discussion following 
\eqref{Exmpl_lon_excursion_joint_cvg},
\[
\left\{{N}\Big|_{[0,W_{1/2})\times (x,\infty]\times[0,\infty]}=0\right\}
= \left\{ \sup_{T_i < W_{1/2}} P_i\leq x \right\}\,,
\]
which is exactly the probability that the largest jump of $1/2$-stable subordinator, before it hits $[1,\infty)$, is less than or equal to $x$. 
\end{proof}

The limiting random variable $V$ in Theorem~\ref{Thm:Exc} has a continuous distribution
with support on $(0,1)$, see  Perman~\cite{Perman93}. This is also the distribution of the longest
completed excursion of the standard Brownian  motion during
time interval $[0,1]$.
Its density  is given in \cite[Corollary 9]{Perman93}. Interestingly, the closed form expression
for the density is known only on the interval $({1}/{3},1)$.

In the special case when all waiting times of CTRW have unit length, the model above boils down to  the simple symmetric random walk on integers. In particular, from \eqref{Exmpl_lon_excursion_joint_cvg} one can deduce  the asymptotic distribution of all upper order statistics for the length of excursions of the simple symmetric random walk  given in Cs\'aki and Hu~\cite{Csaki2003}.

\ack
This work has been supported in part by Croatian Science Foundation under the project 3526.

%
%
%
%
\bibliographystyle{apt}
\bibliography{mybib_initials}{}

\begin{thebibliography}{10}

\bibitem{Anderson87}
{\sc Anderson, K.~K.} (1987).
\newblock Limit theorems for general shock models with infinite mean intershock
  times.
\newblock {\em J. Appl. Probab.\/} {\bf 24,} 449--456.

\bibitem{BB15}
{\sc Basrak, B.}
\newblock Limits of renewal processes and {P}itman--{Y}or distribution.
\newblock
  \url{http://web.math.pmf.unizg.hr/~bbasrak/preprints/pitmanyor_bb2015.pdf}
  2014.
\newblock preprint, uploaded to arxiv.org.

\bibitem{BSSPL}
{\sc Basrak, B. and {\v{S}}poljari{\'c}, D.} (2015).
\newblock Extremes of random variables observed in renewal times.
\newblock {\em Statist. Probab. Lett.\/} {\bf 97,} 216--221.

\bibitem{Berman62}
{\sc Berman, S.~M.} (1962).
\newblock Limiting distribution of the maximum term in sequences of dependent
  random variables.
\newblock {\em Ann. Math. Statist.\/} {\bf 33,} 894--908.

\bibitem{Csaki2003}
{\sc Cs{\'a}ki, E. and Hu, Y.} (2003).
\newblock Lengths and heights of random walk excursions.
\newblock In {\em Discrete random walks ({P}aris, 2003)}.
\newblock pp.~45--52.

\bibitem{deHaan_Ferreira}
{\sc de~Haan, L. and Ferreira, A.} (2006).
\newblock {\em Extreme value theory}.
\newblock Springer Series in Operations Research and Financial Engineering.
  Springer, New York.
\newblock An introduction.

\bibitem{deHaan_Res77}
{\sc de~Haan, L. and Resnick, S.~I.} (1977).
\newblock Limit theory for multivariate sample extremes.
\newblock {\em Z. Wahrscheinlichkeitstheorie und Verw. Gebiete\/} {\bf 40,}
  317--337.

\bibitem{Durrett2010}
{\sc Durrett, R.} (2010).
\newblock {\em Probability: theory and examples} fourth~ed.
\newblock Cambridge Series in Statistical and Probabilistic Mathematics.
  Cambridge University Press, Cambridge.

\bibitem{EKMikosch}
{\sc Embrechts, P., Kl{\"u}ppelberg, C. and Mikosch, T.} (1997).
\newblock {\em Modelling extremal events} vol.~33 of {\em Applications of
  Mathematics (New York)}.
\newblock Springer-Verlag, Berlin.
\newblock For insurance and finance.

\bibitem{Genna06}
{\sc Fa{\"y}, G., Gonz{\'a}lez-Ar{\'e}valo, B., Mikosch, T. and Samorodnitsky,
  G.} (2006).
\newblock Modeling teletraffic arrivals by a {P}oisson cluster process.
\newblock {\em Queueing Syst.\/} {\bf 54,} 121--140.

\bibitem{FellerVol2}
{\sc Feller, W.} (1971).
\newblock {\em An introduction to probability theory and its applications.
  {V}ol. {II}.}
\newblock Second edition. John Wiley \& Sons Inc., New York.

\bibitem{GMS}
{\sc Godr{\`e}che, C., Majumdar, S.~N. and Schehr, G.} (2015).
\newblock Statistics of the longest interval in renewal processes.
\newblock {\em J. Stat. Mech.\/}.
\newblock P03014.

\bibitem{JacodLTSP}
{\sc Jacod, J. and Shiryaev, A.~N.} (2003).
\newblock {\em Limit theorems for stochastic processes} second~ed. vol.~288 of
  {\em Grundlehren der Mathematischen Wissenschaften [Fundamental Principles of
  Mathematical Sciences]}.
\newblock Springer-Verlag, Berlin.

\bibitem{Lamperti61}
{\sc Lamperti, J.} (1961).
\newblock A contribution to renewal theory.
\newblock {\em Proc. Amer. Math. Soc.\/} {\bf 12,} 724--731.

\bibitem{Leadbetter83}
{\sc Leadbetter, M.~R., Lindgren, G. and Rootz{\'e}n, H.} (1983).
\newblock {\em Extremes and related properties of random sequences and
  processes}.
\newblock Springer Series in Statistics. Springer-Verlag, New York.

\bibitem{MS2004}
{\sc Meerschaert, M.~M. and Scheffler, H.-P.} (2004).
\newblock Limit theorems for continuous-time random walks with infinite mean
  waiting times.
\newblock {\em J. Appl. Probab.\/} {\bf 41,} 623--638.

\bibitem{MStoev2009}
{\sc Meerschaert, M.~M. and Stoev, S.~A.} (2009).
\newblock Extremal limit theorems for observations separated by random power
  law waiting times.
\newblock {\em J. Statist. Plann. Inference\/} {\bf 139,} 2175--2188.

\bibitem{Montroll65}
{\sc Montroll, E.~W. and Weiss, G.~H.} (1965).
\newblock Random walks on lattices. {II}.
\newblock {\em J. Mathematical Phys.\/} {\bf 6,} 167--181.

\bibitem{PMM2009}
{\sc Pancheva, E., Mitov, I.~K. and Mitov, K.~V.} (2009).
\newblock Limit theorems for extremal processes generated by a point process
  with correlated time and space components.
\newblock {\em Statist. Probab. Lett.\/} {\bf 79,} 390--395.

\bibitem{Perman93}
{\sc Perman, M.} (1993).
\newblock Order statistics for jumps of normalised subordinators.
\newblock {\em Stochastic Process. Appl.\/} {\bf 46,} 267--281.

\bibitem{PitmanYor97}
{\sc Pitman, J. and Yor, M.} (1997).
\newblock The two-parameter {P}oisson-{D}irichlet distribution derived from a
  stable subordinator.
\newblock {\em Ann. Probab.\/} {\bf 25,} 855--900.

\bibitem{ResEVRVPP}
{\sc Resnick, S.~I.} (1987).
\newblock {\em Extreme values, regular variation and point processes}.
\newblock Springer, New York.

\bibitem{ResHTP}
{\sc Resnick, S.~I.} (2007).
\newblock {\em Heavy-tail phenomena}.
\newblock Springer, New York.

\bibitem{M3ExtremalCCTRW2011}
{\sc Schumer, R., Baeumer, B. and Meerschaert, M.~M.} (2011).
\newblock Extremal behavior of a coupled continuous time random walk.
\newblock {\em Physica A: Statistical Mechanics and its Applications\/} {\bf
  390,} 505 -- 511.

\bibitem{Seneta76}
{\sc Seneta, E.} (1976).
\newblock {\em Regularly varying functions}.
\newblock Lecture Notes in Mathematics, Vol. 508. Springer-Verlag, Berlin.

\bibitem{SS83}
{\sc Shanthikumar, J.~G. and Sumita, U.} (1983).
\newblock General shock models associated with correlated renewal sequences.
\newblock {\em J. Appl. Probab.\/} {\bf 20,} 600--614.

\bibitem{Sibuya60}
{\sc Sibuya, M.} (1960).
\newblock Bivariate extreme statistics. {I}.
\newblock {\em Ann. Inst. Statist. Math. Tokyo\/} {\bf 11,} 195--210.

\bibitem{Skorohod57}
{\sc Skorohod, A.~V.} (1957).
\newblock Limit theorems for stochastic processes with independent increments.
\newblock {\em Teor. Veroyatnost. i Primenen.\/} {\bf 2,} 145--177.

\bibitem{Whitt80}
{\sc Whitt, W.} (1980).
\newblock Some useful functions for functional limit theorems.
\newblock {\em Math. Oper. Res.\/} {\bf 5,} 67--85.

\end{thebibliography}

\end{document}